\documentclass[11pt,reqno, twoside,english]{amsart}
\usepackage[T1]{fontenc}
\usepackage[utf8]{inputenc}
\usepackage{geometry}
\usepackage[english]{babel}
\usepackage{enumitem}
\usepackage{verbatim}
\usepackage{mathrsfs}
\usepackage{amsbsy}
\usepackage{amstext}
\usepackage{amsthm}
\usepackage{amssymb}
\usepackage[all]{xy}
\usepackage{mathrsfs}
\usepackage[normalem]{ulem}
\usepackage{url}            
\usepackage{booktabs}       
\usepackage{amsfonts}       
\usepackage{nicefrac}       
\usepackage{microtype}      
\usepackage[pdftex]{graphicx}
\usepackage{todonotes}
\usepackage{pdfpages}
\usepackage{amsmath}
\usepackage{tikz-cd}
\usepackage{amsfonts}
\usepackage{mathrsfs}  
\usepackage{centernot}
\usepackage{amssymb}
\usepackage{amsthm}
\usepackage{hyperref}
\usepackage{enumitem}
\usepackage[noadjust]{cite}
\usepackage{caption}
\usepackage{dutchcal}
\usepackage{bbm}
\usepackage{tikz}
\usetikzlibrary{decorations.pathreplacing}

\newcommand{\df}{{\, \stackrel{\mathrm{def}}{=}\, }}
\usepackage{mathtools}
\usepackage{stackengine}
\usepackage{amsmath}
\usepackage{tikz}
\usepackage{tikz-cd}
\usepackage{bm}
\usepackage{mathrsfs}
\usepackage{booktabs}
\usetikzlibrary{shapes.geometric}
\usepackage{amscd}
\usepackage{mathabx}
\numberwithin{equation}{section}
\usepackage[original]{imakeidx}
\makeindex 
\usepackage{mathrsfs}
\usepackage{graphics}
\usepackage{eucal}
\usepackage{mathrsfs}
\usepackage{mdwlist}
\usepackage{color}
\usepackage{cleveref}
\crefformat{section}{\S#2#1#3}
\crefformat{subsection}{\S#2#1#3}
\crefformat{subsubsection}{\S#2#1#3}

\makeatletter
\@namedef{subjclassname@2020}{%
  \textup{2020} MSC}
\makeatother

\hypersetup{
	colorlinks   = true, 
	urlcolor     = blue, 
	linkcolor    = purple, 
	citecolor   = blue 
}

\def\XXint#1#2#3{{\setbox0=\hbox{$#1{#2#3}{\int}$ }
\vcenter{\hbox{$#2#3$ }}\kern-.6\wd0}}

\makeatletter
\newcommand*{\rom}[1]{\expandafter\@slowromancap\romannumeral #1@}

\newtheorem{theorem}{Theorem}[section]
\newtheorem{remark}[theorem]{Remark}

\newtheorem{corollary}[theorem]{Corollary}
\newtheorem{proposition}{Proposition}[section]

\newtheorem{lemma}[theorem]{Lemma}

\newtheorem{definition}[theorem]{Definition}

\theoremstyle{definition}

\newcommand{\Z}{\mathbb{Z}}
\newcommand{\R}{\mathbb{R}}

\newcommand{\N}{\mathbb{N}}

\newcommand{\K}{\mathcal{K}}

\newcommand{\F}{\mathbb{F}}

\newcommand{\details}[1]{}

\DeclareMathOperator{\card}{card}

\newcommand{\Exact}{\operatorname{Exact}}

\newcommand{\beq} {\begin{equation}}
\newcommand{\eeq} {\end{equation}}

\usepackage{fancyhdr}
\newcommand{\bracket}[1]{\langle#1\rangle}
\newcommand{\abs}[1]{\lvert#1\rvert}
\newcommand{\norm}[1]{\lVert#1\rVert}

\newcommand{\bv}{\mathbf{v}}
\DeclareMathOperator{\dimh}{\dim_H}

\begin{document}

\title{Exact Approximation in the Field of Formal Power Series}

\author{Aratrika Pandey}
\address{\textbf{Aratrika Pandey} Department of Mathematics, IIT Bombay, Mumbai, India 400065 \\
}
\email{aratrika.pandey@gmail.com,214090004@iitb.ac.in}
\thanks{}

\begin{abstract}
In this article, we determine the Hausdorff dimension of the set of exactly $\psi$-approximable vectors over local fields of positive characteristic. This result is the function field analogue of a recent theorem of Bandi and de Saxcé in the real setting~\cite{bandi2023hausdorff}, and extends the main theorem of Zhang~\cite{MR2834892} to higher dimensions. Our approach adapts the method of Bandi and de Saxcé to the ultrametric setting, enabling us to overcome difficulties arising from the failure of the well-separatedness property for rational functions in higher dimensions.
\end{abstract}
\maketitle
\section{Introduction}
In this article, we investigate the Hausdorff dimension of the set of exactly $\psi$-approximable vectors over local fields of positive characteristic. Our primary objective is to determine the Hausdorff dimension of this set in the ultrametric setting. 

To motivate our investigation, we begin by recalling the classical definition of $\psi$-approximable vectors in $\mathbb{R}^n$. Given a non-increasing function $\psi: \mathbb{R} \to [0, \infty)$, the set $W(\psi)$ of $\psi$-approximable vectors in $\mathbb{R}^n$ is defined as
\[
W(\psi) \df \left\{ x \in \mathbb{R}^n : \left\| x - \frac{p}{q} \right\|_\infty \leq \psi(q) \text{ for infinitely many } (p, q) \in \mathbb{Z}^n \times (\mathbb{Z} \setminus \{0\}) \right\},
\] where $\| \|_{\infty}$ stands for the supremum norm. 
A natural refinement of the classical definition leads to the notions of \emph{badly $\psi$-approximable} and \emph{exactly $\psi$-approximable} vectors. A vector $x \in \mathbb{R}^n$ is said to be badly $\psi$-approximable if it belongs to $W(\psi)$ but not to $W(c\psi)$ for some constant $0 < c < 1$. In particular, the sets of badly and exactly $\psi$-approximable vectors are defined respectively by
\begin{equation}\label{exact definition}
\text{Bad}(\psi)\df W(\psi) \setminus \bigcap_{0 < c < 1} W(c\psi), \quad 
\text{Exact}(\psi) \df W(\psi) \setminus \bigcup_{0 < c < 1} W(c\psi).
\end{equation}

The study of the Hausdorff dimension of sets of exactly $\psi$-approximable vectors has a rich history, tracing back to foundational series of works in the early 21st century~\cite{Bugeaud1, Bugeaud2, Bugeaud3}, that explored various aspects of this topic. Recent research has focused on extending these results to more general settings and higher dimensions. For a detailed account of recent developments, we refer the reader to~\cite{MR4405648, schleischitz2023set, IMRN24, MR4540922}, which include comprehensive references and context.
\smallskip

Let $\mathcal{K}$ denote a local field of characteristic $p > 0$. Without loss of generality, we may assume that $\mathcal{K} = \F_q((X^{-1}))$ , where $q = p^b$ for some integer $b \geq 1$, as every local field of positive characteristic is isomorphic to one of this form \cite[Part I, Chapter I, Theorem 8]{Weil}. We define a non-archimedean absolute value on $\mathcal{K}$ in the following manner, for $h = \frac{f}{g} \in \mathbb{F}_q(X)$ with $f, g \in \F_q[X]$ and $g \ne 0$,
\[
|h| = 
\begin{cases}
q^{\deg f - \deg g}, & \text{if } h \ne 0, \\
0, & \text{if } h = 0,
\end{cases}
\]
and extend this absolute value to the completion $\mathcal{K}$ in the standard way. We denote by $\mathcal{O}$ the ring of integers of $\mathcal{K}$, defined as
\[
\mathcal{O} \df \left\{ x \in \F_q((X^{-1})) : |x| \leq 1 \right\}.
\]
Throughout this paper, we denote the function field by $\K$ and the polynomial ring by $\mathcal{R}$. Let $\psi: \mathbb{R}^+ \to \{q^r:r\in \Z\}$ be a non-increasing function. For $x = (x_1, \ldots, x_n) \in \mathcal{K}^n$, we define the norm $\|x\| \df \max_{1 \le i \le n} |x_i|$. We define the set of $\psi$-approximable vectors in $\mathcal{O}^n$ by
\[ W(\psi)\df\{ x\in \mathcal{O}^n:
\|x - \mathbf{f}/g\| \leq \psi(|g|) \text{ for infinitely many } (\mathbf{f},g) \in \mathcal{R}^n \times \mathcal{R} \setminus \{0\}\}.
\]
 Our main theorem establishes the Hausdorff dimension of the set of \emph{exactly $\psi$-approximable} vectors in $\K^n$. In the ultrametric context, Zhang observed that the set of \emph{exactly $\psi$-approximable} vectors is given by the following definition

 \begin{definition}
\begin{equation}\label{exact}
\Exact(\psi) \df \left\{ x \in \mathcal{O}^n :
\begin{array}{l}
\|x -\mathbf{f}/g\| = \psi(|g|) \text{ for infinitely many } (\mathbf{f},g) \in \mathcal{R}^n \times \mathcal{R} \setminus \{0\}, \\
\|x - \mathbf{f}/g\| \geq \psi(|g|) \text{ for all } |g| \geq |g_0| \text{ for some } g_0 \in \mathcal{R}
\end{array}
\right\}.
\end{equation}
\end{definition}
\begin{remark}
We note that the above definition of $\Exact(\psi)$ in \eqref{exact} is the same as the classical definition of exactly approximable vectors as mentioned in \eqref{exact definition} due to ultrametricity.
\end{remark}

\begin{theorem}\label{main1}
Let  $\psi: \mathbb{R}^+ \to \{q^r:r\in\Z\}$ be a non-increasing function satisfying $\psi(|g|) = o(|g|^{-\frac{n+1}{n}})$. Let $\lambda_\psi$ denote the lower order at infinity of $\psi$, defined by
\[
\lambda_\psi(x) = \liminf_{x \to \infty} \frac{-\log \psi(x)}{\log x}.
\]
Then,
\begin{equation}\label{Hausexcat}
\dimh(\Exact(\psi)) = \frac{n+1}{\lambda_\psi}.
\end{equation}
\end{theorem}
\begin{remark}
We devote the main part of the paper in proving the lower bound of \eqref{Hausexcat}. The upper bound follows from a simple covering argument, which we will prove in Section \ref{sec:upper bound}.
\end{remark}

 The one-dimensional case ($n=1$) of Theorem \ref{main1} was previously studied by Zhang~\cite{MR2834892}, who adapted techniques originally developed by Bugeaud in the setting of real numbers~\cite{Bugeaud1}. 
It is worth noting that the proofs by Zhang and the later geometric approach of Bandi, Ghosh, and Nandi~\cite{MR4540922} crucially rely on the well-separatedness property of rational numbers in dimension one---a feature that is no longer available when $n > 1$. Nonetheless, by adapting the sophisticated methods of~\cite{bandi2023hausdorff}, originally devised for the real case, we are able to establish our result in higher dimensions within the ultrametric framework.

\smallskip

In the context of local fields of positive characteristic, the study of Diophantine approximation was initiated by the seminal work of Mahler~\cite{MR0004272}. Over the past two decades, significant progress has been made in this area through a series of important contributions, including~\cite{MR1777096, AG2, AG3, MR4839473}, which have deepened the understanding of Diophantine approximation phenomena in non-archimedean settings.

\smallskip

Our main theorem complements and strengthens existing results. In particular, a theorem of Kristensen~\cite{MR2006063}, which determines the Hausdorff dimension of $\psi$-approximable sets for specific choices of $\psi$, can be deduced from our main Theorem.
In addition to dimensional results, we also address the non-emptiness of the set of exactly $\psi$-approximable vectors under a mild decay condition on $\psi$.

\begin{theorem}\label{existance}
Let $\psi: \mathbb{R}^+ \to \{q^r:r\in\Z\}$ be a non-increasing function satisfying $\psi(|g|) = o(|g|^{-\frac{n+1}{n}})$. Then,
\[
\Exact(\psi) \neq \emptyset.
\]
\end{theorem}
This result generalizes Theorem~1 of Zhang~\cite[Theorem 1]{MR2834892} to arbitrary dimensions, thereby extending the theory of exact approximation in positive characteristic fields beyond the one-dimensional setting.
\section{Upper Bound}\label{sec:upper bound}

We first prove the upper bound in Theorem \ref{main1}. From \eqref{exact}, if $x \in \Exact(\psi)$, then for infinitely many $(\mathbf{f},g)$ we have,
$$
\|x - \mathbf{f}/g\| = \psi(|g|),
$$
and hence $x \in W(\psi)$. As,
$
\Exact(\psi) \subset W(\psi),
$ so it suffices to bound $\dimh(W(\psi))$.

By definition,
\[
W(\psi)
= \limsup_{g \in \mathcal{R}\setminus\{0\}} \;\bigcup_{\mathbf{f} \in \mathcal{R}^n}
B\!\left(\frac{\mathbf{f}}{g}, \psi(|g|)\right).
\]

Fix $k = \deg g$, so that $|g| = q^k$. The number of polynomials $g$ of degree $k$ is $\asymp q^k$, and for each such $g$, the number of relevant $\mathbf{f}$ is $\asymp |g|^n = q^{kn}$. Hence, the total number of balls at level $k$ is $\asymp q^{k(n+1)}$, each of radius $\psi(q^k)$.

Let $s > 0$. Then the $s$-dimensional Hausdorff content is bounded by
\[
\sum_{k=1}^\infty q^{k(n+1)} \, \psi(q^k)^s.
\]

By the definition of $\lambda_\psi$, for any $\varepsilon > 0$ and all sufficiently large $k$,
\[
\psi(q^k) \le q^{-k(\lambda_\psi - \varepsilon)}.
\]
Thus,
$
q^{k(n+1)} \psi(q^k)^s
\le
q^{k(n+1 - s(\lambda_\psi - \varepsilon))}.
$

The later series converges if
$
n+1 - s(\lambda_\psi - \varepsilon) < 0$, that is if 
$s > \frac{n+1}{\lambda_\psi - \varepsilon}.$
Letting $\varepsilon \to 0$, we obtain
\[
\dimh(W(\psi)) \le \frac{n+1}{\lambda_\psi}.
\] This completes the proof.

\section{Preliminaries}\label{Preliminaries}
 We first briefly recall the definitions of the Hausdorff measure and dimension in any metric space $(Y,d)$. For $ s\geq 0$ the  Hausdorff $s$-measure of a set $A \subset Y^{n}$ is defined as
\begin{equation*}
    \mathcal{H}^{s}(A)\df\lim_{\rho \to 0^{+}} \inf \left\{ \sum_{i=1}^\infty |A_{i}|^{s} : A \subset \bigcup_{i=1}^\infty  A_{i} \quad \text{ and }\quad  |A_{i}|<\rho \, \, \forall \, \, i \right\},
\end{equation*}
where $|\cdot |$ denotes the diameter of a set. The {\it Hausdorff dimension} of set $A \subset Y^{n}$ is defined as
\begin{equation*}\label{Hausdorff}
    \dimh A \df \inf \left\{ s \geq 0 : \mathcal{H}^{s}(A)=0 \right\}\, .
\end{equation*}

Consider a vector $ \mathbf{x} = (x_1, \dots, x_n) $ in $ \K^n $. To this vector we associate the unipotent matrix,
\[
u_x \df \begin{pmatrix}
1 \\
-x_1 & 1 \\
\vdots & & \ddots \\
-x_n & & & 1
\end{pmatrix}.
\]
We also define the diagonal matrix $ g_t $ by
\[
g_t \df \begin{pmatrix}
X^{-nt} \\
  & X^t \\
  & & \ddots \\
  & & & X^t
\end{pmatrix},
\]
where $t\in \N$. For a rational function $ v = \left( \frac{f_1}{g}, \dots, \frac{f_n}{g} \right) \in \F_q(X)^{n} $, we associate the vector $ \mathbf{v} = (g, f_1, \dots, f_n) $ with $ \mathcal{R}^{n+1}$. We define the height of the vector $v \in \F_q(X)^{n+1}$ by $H(v)=\text{max}\{|g|,|f_1|,\cdots,|f_n|\}$. The space $\K^{n+1}$ is equipped with a norm given by the largest absolute value among its coordinates:
$
\|\mathbf{v}\| = \max_{1\leq i \leq n+1} \left| v_i \right|$. Now, we recall the ultrametric inequality.
\begin{lemma}[Ultrametric inequality]\label{ultrametric}
For every $v,w\in \K^n$, we have
\begin{equation}\label{1}\|v+w\|\leq \max\{\|v\|,\|w\|\}.\end{equation}
Moreover, the inequality in \eqref{1} is an equality if $\|v\|\neq \|w\|$.
\end{lemma}
We now state Dani’s correspondence \cite{Dani} in our context. For further details on the general proof of this correspondence in the ultrametric setting, we refer the reader to \cite{MR2321374} and \cite{MR2961287}.
\begin{remark}
Let $\Psi(s)=\psi(s)^{-\frac{n}{n+1}}$. Throughout the remainder of the paper,
we assume that $\Psi(s)\in q^{n\mathbb{Z}}$. Without this assumption, all
subsequent arguments continue to hold up to a constant factor, which would
need to be tracked throughout. To avoid this additional bookkeeping, we make
the simplifying assumption that $\Psi(s)\in q^{n\mathbb{Z}}$.
\end{remark}
\begin{proposition}[Dani's correspondence in the ultrametric setting]
\label{Dani}
Let $ \psi:\R^{+}\to \{q^{-(n+1)t}:t\in \N\}$ be a decreasing function and define $ \Psi(s) \df \psi(s)^{-\frac{n}{n+1}}$. The following holds,

\begin{enumerate}[label=(\alph*)]
    \item Suppose $ d(x, v) \leq \psi(H(v))$ for some $ v \in \F_q(X)^n \cap \mathcal{O}^n $, and suppose $ t $ satisfies $ q^{nt} = \Psi(H(v))$. Then,
    \[
    \| g_t u_x \mathbf{v} \| \leq q^{-nt} \Psi^{-1}(q^{nt}) \quad\text{and}\quad
    \abs{\bracket{(g_tu_x\bv)_1}}=\norm{g_tu_x\bv}.
    \]
    
    \item If $ \| g_t u_x \mathbf{v} \| \leq q^{-nt} \Psi^{-1}(q^{nt}) $ and $\abs{\bracket{(g_tu_x\bv)_1}}=\norm{g_tu_x\bv}$ for some $\bv\in\mathcal{R}^{n+1}$ then one has,
    \[
    H(v) \leq \Psi^{-1}(q^{nt}),\quad\text{and} \quad d(x, v) \leq \psi(H(v)).
    \]
    Moreover, if one assumes that the function $ \alpha(H) = H \psi(H) $ is non-increasing instead of the condition $\abs{(g_tu_x\bv)_1}=\norm{g_tu_x\bv}$ for some $\bv\in\mathcal{R}^{n+1}$, then there exists a rational function $ v \in \F_q(X)$ such that the last conclusion holds that is,
    \[
    H(v) \leq \Psi^{-1}(q^{nt}),\quad\text{and} \quad d(x, v) \leq \psi(H(v)).
    \]
\end{enumerate}
\end{proposition}

\begin{proof}
For part (a), assume $ d(x, v) \leq \psi(H(v)) $. Let $\Psi\in q^{n\Z}$ and $ t$ be such that
\[
q^{nt} = \Psi(H(v)).
\]
Using the definition of the norm $ \| g_t u_x \mathbf{v} \| $, we have
\[
\| g_t u_x \mathbf{v} \| = \max \left\{ q^{-nt} H(v), q^t H(v) d(x, v) \right\}.
\]
From the assumption $ d(x, v) \leq \psi(H(v))$, it follows that
\[
q^t H(v) d(x, v) \leq q^t H(v) \psi(H(v)) \leq q^t H(v) \Psi(H(v))^{-\frac{n+1}{n}}.
\]
Since $ q^{nt} = \Psi(H(v)) $, we obtain
\[
q^t H(v) \psi(H(v)) \leq q^{-nt} \Psi^{-1}(q^{nt}).
\]
Thus, we conclude that
$$
\| g_t u_x \bv\| \leq q^{-nt} \Psi^{-1}(q^{nt}).$$ and $$|(g_tu_x\bv)_1|=q^{-nt}H(v)=q^{-nt}\Psi^{-1}(q^{nt})=\|g_tu_x\bv\|.$$

For part (b), assume that
$$
\| g_t u_x \mathbf{v} \| = \max \left\{ q^{-nt} H(v), q^t H(v) d(x, v) \right\} \leq q^{-nt} \Psi^{-1}(q^{nt}).$$
Then, it follows that
$
H(v) \leq \Psi^{-1}(q^{nt}).$ Also
the condition $\abs{(g_tu_x\bv)_1}=\norm{g_tu_x\bv}$ yields
\[
q^{t}H(v)d(x,v) \leq q^{-nt}H(v).
\]
And from that we get,
\[
d(x,v) \leq q^{-(n+1)t} \leq \Psi(H(v))^{-\frac{n+1}{n}} = \psi(H(v)).
\]
Now if we assume $\alpha(H) = H \psi(H)$ is non-increasing instead of the assumption $\abs{(g_tu_x\bv)_1}=\norm{g_tu_x\bv}$ for some $\bv\in\mathcal{R}^{n+1}$, we have
\[
\alpha(\Psi^{-1}(q^{nt})) = \Psi^{-1}(q^{nt}) q^{-(n+1)t} \leq \alpha(H(v)).
\]
This implies that
\[
d(x, v) \leq H(v)^{-1} q^{-(n+1)t} \Psi^{-1}(q^{nt}) \leq \psi(H(v)),
\]
thus proving part (b).
\end{proof}

\begin{corollary}\label{corollary}
Consider the function $ \psi: \R^{+} \to \{q^{-(n+1)t}:t\in \N\} $ such that $ s \mapsto s \psi(s) $ is decreasing and $s^{\frac{n+1}{n}} \psi(s) \to 0 $ as $ s \to \infty $. Define
\[
\Psi(s) = \psi(s)^{-\frac{n}{n+1}},
r_\psi(t) = -nt + \log \Psi^{-1}(q^{nt}).
\]
Assume that $ x \in \K^n $ satisfies the following conditions,

\begin{enumerate}[label=(\alph*)]
    \item For all sufficiently large $ t > 0 $, we have
    $
    \log \lambda_1(g_t u_x \mathcal{R}^{n+1}) \geq r_{c\psi}(t).
    $
    
    \item For arbitrarily large $ t > 0 $, we also have
    $
    \log \lambda_1(g_t u_x \mathcal{R}^{n+1}) \leq r_\psi(t).
    $
\end{enumerate}
Then, $ x \in \Exact(\psi) $.
\end{corollary}

\begin{proof}
According to Proposition~\ref{Dani}, the first condition implies that \( x \notin W(c\psi) \) for all \( c < 1 \). This implies, for all sufficiently large $t>0$, one has  $\log \lambda_1(g_t u_x \mathcal{R}^{n+1}) \geq r_{c\psi}(t).$
Moreover, leveraging the second part of the proposition, we deduce that $x \in W(\psi)$. 
\end{proof}

We are now in a position to write down the template that we want to work with. Recall that in the real case the template has been defined in \cite{MR4671568} as a piecewise affine function with certain properties that approximates the minimum function on unimodular lattices. The analogous notion of a template in the case of a field of formal power series has not yet been developed. We use the parametric geometry of numbers in function fields developed by Roy and Waldschmidt \cite{roy2017parametric} and redefine the notion of a template in our case accordingly. 
\subsection{\texorpdfstring
  {The template of the first lattice minimum $\lambda_1$}
  {The template of the first lattice minimum lambda1}}
Let $\Lambda$ be a lattice in $\mathcal{K}^n$. We define 
\[
\lambda_1(\Lambda)=\inf\{\|v\|\,| v\in \Lambda\setminus \{0\}\}.
\]
We take $\Lambda=g_tu_x\mathcal{R}^{n+1}$ in what follows.
Using the assumptions that $s\to s\psi(s)$ is decreasing and that $\Psi(s)=\psi(s)^{-\frac{n}{n+1}}$ we obtain the following conditions:\\
\begin{itemize}
\item  $t \to r_{\psi}(t)-t$ is decreasing,
\item  $t \to r_{\psi}(t) + nt$ is increasing, and
\item  $\lim_{t\to\infty} r_{\psi}(t) = -\infty$. 
\end{itemize}
For each $x\in \K^n$, define
\[
\begin{array}{cccc}
c_x\colon & \mathbb{Z}^{+} & \to \mathbb{Z}\ \text{by}\\
& t & \mapsto & c_x(t)=\log\lambda_1(g_tu_x\mathcal{R}^{n+1}).
\end{array}
\]
By Corollary \ref{corollary}, to establish Theorem \ref{existance} under the extra assumption that $s \mapsto s\psi(s)$ is decreasing, it is enough to find a point $x$ in $\K^n$ such that the function $c_x$ meets both conditions.
\begin{equation}\label{c_condition}
\left\{
\begin{array}{ll}
 \forall t>0\ \mbox{sufficiently large},\ c_x(t) \geq r_{c\psi}(t),\,\text{and}\\
\exists t>0\ \mbox{arbitrarily large}:\ c_x(t)=r_\psi(t).
\end{array}
\right.
\end{equation}
The parametric geometry of numbers in the real setting, as developed by Schmidt and Summerer \cite{SchmidtSummerer}, offers a combinatorial framework for describing $c_x$. Similarly, in the case of the function field, Roy and Waldschmidt \cite{roy2017parametric} established an analogous combinatorial approach for $c_x$. In particular, this implies the existence of a continuous piecewise affine function $T_x$ with slopes in $\{-n,0,1\}$, which ensures that the difference $c_x - T_x$ remains bounded in $\Z^{+}$. Conversely, if one begins with such a template $T$, then there exist points $x \in \K^n$ for which $c_x$ stays within a bounded distance of $T$. The converse statement in the real case was established by Roy \cite{roy2015schmidt}.

\smallskip
We now take such an increasing integer sequence $\{t_k\}_{k\geq 1}$ that tends to infinity sufficiently fast and let
$t_k^{-} = t_k + \frac{r_{\psi}(t_k)}{n}$ and $t_k^{+} = t_k-r_{\psi}(t_k)$.
Provided that $(t_k)$ increases fast enough, one always has 
\[
0 < \lfloor t_1^-\rfloor < t_1 < \lceil t_1^+\rceil < \lfloor t_{2}^- \rfloor< \ldots,
\]
and we define a function $T$  with slopes in $\{-n,0,1\}$ which we call as a template, by $T (0) = 0$ and
\begin{equation}\label{slope}T(t+1)-T(t)=\begin{cases}
 0\ \text{if}\  \lceil{t_{k-1}^{+}}\rceil < t < \lfloor{t_k^{-}}\rfloor,\\
-n\ \text{if}\  \lceil{t_k^{-}}\rceil <t<t_{k},\text{and}\\
 1 \ \text{if}\  t_k<t< \lfloor{t_k^{+}}\rfloor.
\end{cases}
\end{equation}
Note that this function satisfies $T(t_k) = r_{\psi}(t_k)$ for each $k \geq 1$ whenever $t_k\in\mathbb{Z}$. Essentially, we are only taking the integer points that lie on the dotted line with slopes $\{0,-n,1\}$.

\begin{figure}[h!]
\begin{center}
\begin{tikzpicture}
\draw (0,0) node[above left] {0};
\draw[->] (-1., 0) -- (10, 0) node[right] {$t$};
\draw[<-] (0, .5) -- (0, -3);

\draw[color=black, dotted] (0,0) .. controls (1,-1) and (2,0) .. (3,-1);
\draw[color=black, dotted] (3,-1) .. controls (4,-2) and (5,-2) .. (6,-2);
\draw[color=black, dotted] (6,-2) .. controls (7,-2) and (8,-3) .. (10,-3);
\draw[color=blue] (10,-3) node[above] {$r_\psi(t)$};
\draw[color=red, dotted] (0,0) -- (1.5,0) -- (2,-.5);
\draw[color=red, dotted] (2,-.5) -- (3,0) -- (4,0) -- (6,-2) -- (9,-.5);
\draw[color=red] (9,-.5) node[right] {$T(t)$};

\draw (3,-.1) -- (3,.1) node[above] {$\lceil{t_{k-1}^+}\rceil$};
\draw (4,-.1) -- (4,.1) node[above] {$\lceil t_{k}^-\rceil$};
\draw (6,-2) -- (6,.1) node[above] {$t_{k}$};

\end{tikzpicture}
\end{center}
\caption{The template $T$ above the graph of $r_\psi$ where the lines are dotted to indicate that we are only interested in the discrete values of $t$.}
\label{T_and_rpsi}
\end{figure}
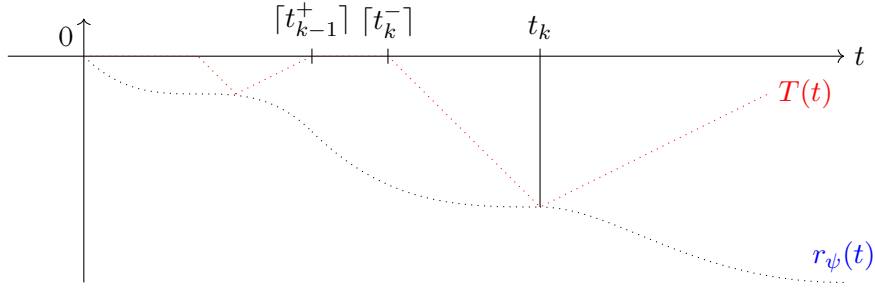
\smallskip
We create a template $T$ that meets the conditions of \eqref{c_condition} and then find points in the Cantor set inside $\Exact(\psi)$ that closely follow this model's path.
\section{\texorpdfstring
  {Cantor Set Construction in $\Exact(\psi)$}
  {Cantor Set Construction in Exact(ψ)}}\label{section three}
We construct a Cantor set $K_{\infty}$ using the same technique as in \cite{bandi2023hausdorff}, such that the Hausdorff dimension of $K_{\infty}$ is $\frac{n+1}{\lambda_{\psi}}$. As a result, we get the required lower bound for the Hausdorff dimension of $\Exact(\psi)$. Here, we note that our argument deviates from \cite{bandi2023hausdorff} at places because of the ultrametric setting.
\subsection{\texorpdfstring
  {Properties of the Cantor Set $K_{\infty}$}
  {Properties of the Cantor Set K∞}}
We assume that the sequence $(t_k)$ satisfies
\[
\lim_{k\to\infty}\frac{r_\psi(t_k)}{nt_k} = \limsup_{t\to\infty}\frac{r_\psi(t)}{nt} \df -\gamma_\psi
\]
where $\Psi(s)=\psi(s)^{-\frac{n}{n+1}}$ and 
$
r_\psi(t) = -nt + \log \Psi^{-1}(q^{nt}).
$\\
Let
\[
M_k = -\sup_{t\geq t_{k-1}} r_\psi(t).
\]
Observe that this definition ensures $M_k \leq -r_\psi(t_{k-1})$.

The time gap between $t_{k-1}$ and $t_k$ is sufficiently large, making $M_k$ relatively small in comparison to $-r_\psi(t_k)$. This parameter $M_k$ will later be used to establish small intervals around $t_k$ or $r_\psi(t_k)$.

Moving forward, we introduce three constants, as referenced in \cite{bandi2023hausdorff}:
\begin{itemize}
\item $R_0 \geq 1$, which depends solely on $n$;
\item $R_1$, determined by $\gamma_\psi$ and $R_0$;
\item $R_2$, which depends on $n$, $R_0$, and $R_1$.
\end{itemize}
We recall that $t_k^- < t_k$ and $t_k^+ > t_k$, as defined in Section \ref{Preliminaries},
\[
t_k^- = t_k + \frac{r_\psi(t_k)}{n}
\quad\mbox{and}\quad
t_k^+ = t_k + R_2 M_k.
\]
\color{black}
Here, we consider a discretized flow within the framework of the field of formal power series. However, whether $t_k^{-}$ or $t_k^{+}$ belongs to $\mathbb{Z}^{+}$ is not significant for the construction of the Cantor set. We simply take the floor or ceiling of these values as needed, depending on the circumstances.
We now focus on building a Cantor set $K_{\infty}$ consisting of points $x$ for which the trajectory $c_x$ meets the following two criteria, drawing inspiration from the construction in \cite{bandi2023hausdorff}, but adapted for an ultrametric space.
\vspace{0.4cm}
\begin{enumerate}[label=(\Alph*)]
\item \label{a}
    For all $t\in[t_{k-1}^+, t_k^--4R_0M_k]\cap\mathbb{Z}^{+}$, $c_x(t) \geq -M_{k}$; and for all $t\in (t_k^{-}-4R_0M_k,t_k^{-})\cap\mathbb{Z}^{+}$, $c_x(t)\geq -5R_0M_k$
\item \label{b}
    For each $k$, there exists $\bv_k=\bv_k(x)\in\mathcal{R}^{n+1}$ such that
    \[
   \quad \forall t\in [t_{k}^-,t_{k}^+]\cap \mathbb{Z}^{+},\quad   c_x(t) = \log \norm{g_tu_x \bv_k}
    \]
    and moreover, the point $v_k$ in $\F_q(X)^n$ corresponding to $\bv_k$ satisfies
    \[
    \left\{\begin{array}{ll}
	(i) &    q^{\lfloor nt_k^-\rfloor-5R_0M_k} \leq H(v_k) \leq q^{\lceil nt_k^-\rceil-3R_0M_k}\\
	(ii) &   d(v_k,x) = \psi(H(v_k))\ 
    \\
	(iii) &    \lfloor{t_k-R_1M_k\rfloor} < t_k^x\df-\frac{1}{n+1}\log d(v_k,x) < \lceil{t_k+1}\rceil.
    \end{array}\right.
    \]
\end{enumerate}
\begin{lemma}\label{tkminusr}
Let $x$ in $\K^n$ be such that $c_x$ satisfies \ref{a} and \ref{b} for all $k\geq 1$.
Then $x\in \Exact(\psi)$.
\end{lemma}
\begin{proof}
From the construction, we have $d(v_k,x)=\psi(H(v_k))$ for all $k$, so $x$ lies in $W(\psi)$.
Now, we show for any $0<c<1$, the vector $x$ does not belong to $W(c\psi)$.
We use the exact form of Dani's correspondence as in \cite{bandi2023hausdorff} to show that for all $t > 0$ large enough, if $\mathbf{v} \in \mathcal{R}^{n+1}$ satisfies 
$
\left|( g_t u_x \mathbf{v})_1 \right| = \| g_t u_x \mathbf{v} \|,$
then
\[
\| g_t u_x \mathbf{v} \| \geq q^{-nt} \Psi^{-1}(q^{nt}).
\]

Essentially, one needs to prove that if $c < 1$, then $x$ lies in $W(\psi)$ but not in $W(c\psi)$. If $t$ belongs to some interval $[t_k^+, t_k^- - 4R_0M_k] \cap \mathbb{Z}^{+}$, this follows from \ref{a} and the definition of $M_k$, since
\[
c_x(t) \geq -M_k \geq r_\psi(t).
\]
On the interval $(t_k^- - 4R_0M_k, t_k^-) \cap \mathbb{Z}$, it follows from \ref{a} that
\[
c_x(t) \geq -5R_0M_k \geq r_{\psi}(t)
\]
provided the sequence $(t_k)$ is chosen to increase sufficiently fast in order to ensure that
\[
\sup_{t \geq \lceil t_k^- - 4R_0M_k \rceil} r_\psi(t) < -5R_0M_k.
\]
This is feasible because $r_\psi(t)$ approaches $-\infty$ as $t_k$ increases towards $+\infty$, assuming that $\gamma_\psi < 1$. Now, for $t \in [t_k^-, t_k^+] \cap \mathbb{Z}$, the following expression holds,
\[
c_x(t) = -nt + \log H(v_k) + \max \left( (n+1)t + \log d(v_k,x) \right).
\]
Consequently, the condition
$
\left| (g_t u_x \mathbf{v}_k)_1\right| = \| g_t u_x \mathbf{v}_k \|
$
is satisfied only if
\[
t \leq -\frac{1}{n+1} \log d(v_k,x).
\]
For such an integer $t$, we have
\[
c_x(t) = -nt + \log H(v_k).
\]
Since $d(v_k,x) = \psi(H(v_k)) \geq c\psi(H(v_k))$ for any constant $c < 1$, it follows that
\[
t \leq -\frac{1}{n+1} \log \left( \psi(H(v_k)) \right)
\]
which implies
\[
H(v_k) \geq \Psi^{-1}(q^{nt}).  
\]
Thus, we derive the inequality
\[
c_x(t) \geq -nt + \Psi^{-1}(q^{nt}) = r_{\psi}(t).
\]
This establishes that the condition
\[
\left| (g_t u_x \mathbf{v}_k)_1 \right| = \| g_t u_x \mathbf{v}_k \|
\]
implies
\[
\| g_t u_x \mathbf{v}_k \| \geq q^{-nt} \Psi^{-1}(q^{nt}).
\]
Furthermore, for any vector $\mathbf{v}$ that is linearly independent of $\mathbf{v}_k$, we can apply Minkowski's second theorem to obtain the following bound, for $t\in [[t_k^{-}] \leq t \leq t_k^x \df -\frac{1}{n+1} \log d(v_k,x)]\cap\Z$,
\begin{align*}
\log \| g_t u_x \mathbf{v} \| & \geq \log \| g_{[t_k^-]} u_x \mathbf{v} \| + (t - [t_k^-]) - O_n(1)\\ &\geq \log \| g_{[t_k^-]} u_x \mathbf{v} \| + (t - t_k^{-}) - O_n(1) \\
& \geq -5R_0M_k - O_n(1) \\
& \geq r_\psi(t).
\end{align*}
We observe that condition \ref{b}, along with the definition of $t_k^+$, leads to the inequality $t_k^+ - t_k^x \leq (R_1 + R_2) M_k$. Therefore, for integer values of $t$ in $t_k^x \leq t \leq t_k^+$, we have the following sequence of inequalities:
\begin{align*}
\log \| g_t u_x \mathbf{v} \| \geq \log \| g_{\lceil t_k^x \rceil} u_x \mathbf{v} \| - n(t - \lceil t_k^x \rceil)
\geq -5R_0 M_k - O_n(1) - n(R_1 + R_2) M_k
\geq r_\psi(t),
\end{align*}
under the assumption that the sequence $(t_k)$ grows sufficiently fast to guarantee that
\[
-(5R_0 + n(R_1 + R_2)) M_k - O_n(1) \geq r_\psi(t)
\]
for all $t \geq \lfloor t_k - R_1 M_k \rfloor$. This establishes that for any $c < 1$, it follows that $x \notin W(c\psi)$, which in turn implies $x \in \text{Exact}(\psi)$.
 \end{proof}

\subsection{Construction of the Cantor set}
\label{ss:ccs}
We set $K_0={\mathcal{O}}^n$ which we call a cube. Then we fix some large integer $M>0$ such that $N=q^{(n+1)M}$ is an integer. The Cantor set $K_\infty$ is defined as the decreasing intersection
\[
    K_{\infty} \df \bigcap_{l=0}^{\infty} K_l,
\]
where each $K_l$, referred to as the $l$-th \textit{level} of the Cantor set, consists of a finite collection of disjoint cubes, each with a side length of $N^{-l}$.

Each set $K_l$ is constructed inductively to ensure that, for all $x$ in $K_l$, the conditions~\ref{a} and \ref{b} hold up to the time $t = [lM]$.  

More specifically, we will verify that for arbitrarily large $l$, for any $x \in K_l$ and all $k \geq 1$, the following holds:
\begin{enumerate}[label=(\Alph*$_1$)]
\item \label{A_1}
    For all $t\in [0,lM]\cap [t_{k-1}^+, t_k^--4R_0M_k]\cap\mathbb{Z}^{+}$, then $c_x(t) \geq -M_{k}+M$;
\item \label{B_1}
    For each $k$, there exists $\bv_k=\bv_k(x)\in\mathcal{R}^{n+1}$ such that
    \[
   \quad \forall t\in [0,lM]\cap [t_{k}^-,t_{k}^+]\cap\mathbb{Z},\quad   c_x(t) = \log \norm{g_tu_x \bv_k},
    \]
    and the point $v_k$ in $\F_q(X)^n$ corresponding to $\bv_k$ satisfies
    \[
    \left\{\begin{array}{ll}
	(i) &    q^{\lfloor nt_k^-\rfloor-5R_0M_k} \leq H(v_k) \leq q^{\lceil nt_k^-\rceil-3R_0M_k},\\
	(ii) &  d(v_k,x) = \psi(H(v_k)), \text{and}
 \\
	(iii) &    \lfloor t_k-R_1M_k\rfloor < t_k^x=-\frac{1}{n+1}\log d(v_k,x) < \lceil t_k + 1\rceil.
    \end{array}\right.
    \]
\end{enumerate}

We first claim that condition (iii) can be deduced from (i) and (ii), provided that the sequence $t_k$ and the constant $R_1$ are selected appropriately. Next, we claim that such a choice of $t_k$ is indeed feasible. Since we have assumed that $\lambda_\psi < +\infty$, which can be rewritten as
\[
\gamma_{\psi} \df \liminf -\frac{1}{nt} r_\psi(t) < 1.
\]

\begin{lemma}\label{c1}
Let
$
R_1 = \frac{10R_0}{1-\gamma}
$
and assume that $t_k$ is chosen so that 
\begin{equation}\label{tkminus}
 r_\psi(\lfloor t_k-R_1M_k\rfloor) \leq r_\psi(t_k) + \frac{1+\gamma}{2}R_1M_k.
\end{equation}
Then condition (iii) from (B$_l$) above is implied by (i) and (ii).
\end{lemma}

\color{black}
\begin{proof}
This condition is nothing but the control of $r_{\psi}$ in small intervals around the $t_k'$s using the slope condition. The proof follows mimicking the argument in \cite[Lemma 2]{bandi2023hausdorff} with the obvious modification in $\mathcal{K}=\F_q((X^{-1}))$, and hence we skip the proof here.
\details{
Indeed,
$\Psi(H(v_k))\leq d(x,v_k)^{-\frac{n}{n+1}}$
so
$\log\Psi(H(v_k)) \leq t_k^x$
and
$-t_k^x+\log H(v_k) \leq -t_k^x + \log \Psi^{-1}(q^{t_k^x}) = r_\psi(t_k^x)$.
}
\end{proof}
We note that the times $t_k$ can be chosen inductively to ensure that \eqref{tkminus} is always satisfied. To control $r_{\psi}(t_k)$ both after $t_k$ and shortly before time $t_k$, we now state the following lemma from \cite[Lemma 3]{bandi2023hausdorff}, which remains valid in $\K$ in a similar way.
\begin{lemma}\label{finalchoicetk}
Assume $\gamma_\psi=\gamma<1$ and $t_{k-1}$ have been defined.
Given $R>0$ (possibly depending on $t_{k-1}$), we may always choose $t_k$ arbitrarily large so that
\begin{equation}\label{limsup2}
-\gamma_\psi - \frac{1}{k} \leq \frac{r_\psi(t_k)}{nt_k} \leq -\gamma_\psi + \frac{1}{k}
\end{equation}
and
\begin{equation}\label{minusr2}
r_\psi(t_k-R) \leq r_\psi(t_k)+\frac{1+\gamma}{2}R
\end{equation}
and for all $t\geq t_k$,
\begin{equation}\label{aftertk2}
r_\psi(t) \leq r_\psi(t_k) + R.
\end{equation}
\end{lemma}
We now establish a lemma concerning the control of the height of rational functions within the cubes $C \subset K_{l_k}^{-}$.
\begin{lemma}
\label{rational functions near badly approximable points}
Given the parameters $t\in \mathbb{N}$ and $R\gg 1$, assume that $x\in\mathcal{O}^{n}$ is such that $\lambda_1(g_tu_x \mathcal{R}^{n+1}) \geq q^{-R}$.
Then there exists a rational function $v$ such that,
\begin{itemize}
\item $q^{nt-R} \leq H(v) \leq q^{n(t+2nR)}$ \ \text{and}
\item $d(x,v) \leq \frac{1}{2}q^{-(n+1)t}$.
\end{itemize}
\end{lemma}
\begin{proof}
Applying Minkowski's first theorem to the unimodular lattice $g_{t+2nR}u_x\mathcal{R}^{n+1}$, there exists  $\bv$ in $\mathcal{R}^{n+1}$ such that
\[
\norm{g_{t+2nR}u_x\bv} \leq 1.
\]
Let $v$ denote the function in $\K^n$ corresponding to $\bv$. Since
\[
\norm{g_{t+2nR}u_x\bv} \leq 1, 
\]
one has
\[
H(v)\text{max}(q^{-n(t+2nR)},q^{t+2nR}d(x,v))\leq 1 
,\]
which implies
\begin{equation}
\quad H(v)\leq q^{n(t+2nR)}
\quad\mbox{and}\quad d(x,v) \leq q^{-t-2nR}H(v)^{-1}.\label{bad}
\end{equation}
Now, since
\[
\lambda_1(g_tu_x\mathcal{R}^{n+1})\geq q^{-R},
\]
we see that
\[
\norm{g_tu_x\bv} = H(v) \max(q^{-nt},q^{t} d(x,v)) \geq q^{-R}.
\]
Using $d(x,v) \leq q^{-t-2nR}H(v)^{-1}$, we have
\[
H(v)\geq q^{nt-R}.
\]
Now using \eqref{bad}, we get
\[
d(x,v) \leq q^{-{t}-2nR} q^{-nt+R} < \frac{1}{2}q^{-(n+1)t}, \ \text{provided $R$ large enough.}
\] 
\end{proof}

We now revisit the construction of our Cantor set. Define the initial set as $K_0 = \mathcal{O}^n$, referred to as a cube. Suppose that $K_{l-1}$ has already been established in such a way that it satisfies the conditions (A$_{l-1}$) and (B$_{l-1}$).  

Selecting any cube $C$ within $K_{l-1}$, we partition it into $N^n$ smaller subcubes, each having a sidelength of $N^{-l}$. The Cantor set is then built by choosing a subset of these subcubes to form the $l$-th level, denoted as $K_l$. We use $K_l(C)$ to represent the collection of these selected subcubes.  

Setting $R_0 = 4n^2$, $R_1 = \frac{10R_0}{1 - \gamma}$, and $R_2 = R_1 + 6R_0 + C_1 + 1$, we proceed with the definition.

\[
l_k^- = \lceil\frac{\lfloor t_k^--4R_0M_k\rfloor}{M}\rceil,
\quad\mbox{and}\quad
l_k^+ = \lfloor\frac{t_k+R_2M_k}{M}\rfloor.
\]
We now have two cases,

\begin{description}
    \item[Case 1] $l_{k-1}^{+}< l \leq l_{k}^-$.\\
Set $K_{l}(C)$ to be the set of subcubes $C'\subset C$ such that for all $x$ in $C'$,
\[
c_x([lM]) \geq -M_k + M.
\]
    \item[Case 2] $l_{k}^- < l\leq l_{k}^+$.\\
Let $x_k$ denote an element of the unique cube $C_{l_k^-}$ of level $l_k^-$ containing~$C$, and note that $\lambda_{1}(g_{l_k^-M}u_{x_k}\mathcal{R}^{n+1})\geq q^{-M_{k}}$.
By Lemma~\ref{rational functions near badly approximable points} above applied at time $t=l_k^-M$ and with parameter $R=2M_k$, there exists a rational function $v_k\in C_{l_{k}^-}$ such that 
\begin{equation*} 
    q^{nl_{k}^-M-2M_{k}} \leq H(v_k) \leq q^{n(l_{k}^-M+4nM_{k})}.
\end{equation*}
With our choice of $R_0$ and the definition of $l_k^-$, this implies
\[
q^{\lfloor nt_k^-\rfloor-5R_0M_k} \leq H(v_k) \leq q^{\lceil nt_k^-\rceil-3R_0M_k}.
\]
Pick $y_{k}\in C_{l_{k}^-}$ such that
\begin{equation*}
    d(v_k,y_{k})=\psi(H(v_k)).
\end{equation*}
For each $l \in \{l_k^-,\dots,l_k^+\}$, take
\[
K_{l}(C)=\{C_{l}(y_k)\}
\]
where $C_l(y_k)$ denotes the cube of level $l$ containing $y_k$.
\end{description}
We now use induction to verify that if $ K_l(C) $ is chosen as described above, then the conditions $ \text{A}_1 $ and $ \text{B}_1 $ hold true for sufficiently large values of $ l $.

\begin{description}
\item[Case 1]
The condition \ref{A_1} is clearly satisfied by the selection of $ K_l(C) $, and \ref{B_1} holds since it is same as $ \text{B}_{l-1} $.
\item[Case 2]
Assuming that both $ \text{A}_{l_k^-} $ and $ \text{B}_{l_k^-} $ are satisfied for $ C $ in $ K_{l_k^-} $, we aim to show that the conditions $ \text{A}_{l_k^+} $ and $ \text{B}_{l_k^+} $ hold for every $ x \in K_{l_k^+}(C) = C_{l_k^+}(y_k) $.

We begin by proving $ \text{B}_{l_k^+} $. From the construction, we have,
\[
q^{\lfloor nt_k^- \rfloor - 5R_0M_k} \leq H(v_k) \leq q^{\lceil nt_k^- \rceil - 3R_0M_k}.
\]
Let $ \bv_k $ be the vector in $ \mathcal{R}^{n+1} $ corresponding to the rational function $ v_k $ in $ C_{l_k^-} $. For any point $ x \in C_{l_k^+}(y_k) $, applying the ultrametric inequality as in Lemma \ref{ultrametric}, we obtain,
\[
d(x, v_k) \leq \max\{d(x, y_k), d(y_k, v_k)\} \leq \max\{q^{-(n+1)l_k^+M}, \psi(H(v_k))\}.
\]
Using the fact that $ H(v_k) \leq q^{nt_k^-} = \Psi^{-1}(q^{nt_k}) $, we deduce,
\[
q^{-(n+1)t_k} \leq \psi(H(v_k)).
\]
Since $ l_k^+ M \geq t_k + M_k $ and $q^{-M_k}\ll 1$ (provided we choose $ t_k $ sufficiently large such that $ q^{-M_k} $ is small enough), this leads to,
\[
 q^{-(n+1)l_k^+M} \ll \psi(H(v_k).
\]
Thus, we have due to ultrametric equality \ref{ultrametric},
\[
d(x, v_k) = \psi(H(v_k)).
\]

At this stage, conditions (i) and (ii) of $ \text{B}_{l_k^+} $ are satisfied. By applying Lemma \ref{rational functions near badly approximable points} and using our choice of the sequence $ (t_k) $, the third condition is satisfied. We now claim the following,

\medskip
\noindent\textbf{Claim:} The vector $ \bv_k $ achieves the first minimum of $ g_t u_x \mathcal{R}^{n+1} $ for $ t \in [t_k^-, t_k^+] \cap \mathbb{Z}.$
\begin{proof}[\textbf{Proof of the Claim}]
At time $ T = \lfloor t_k^- - 4R_0M_k \rfloor $, we know that,
$
\lambda_1(g_T u_x \mathcal{R}^{n+1}) \geq q^{-M_k}.
$
By Minkowski's second theorem, we also have,
\[
\lambda_{n+1}(g_T u_x \mathcal{R}^{n+1}) \lesssim q^{nM_k}.
\]
Using the fact that the largest eigenvalue of $ g_s $ is $ q^s $ for all integer $ s \geq 0 $, it follows that,
\[
\lambda_{n+1}(g_{T+s} u_x \mathcal{R}^{n+1}) \lesssim q^{nM_k + s}.
\]

Now, for $ T+s \leq \lfloor t_k^x \rfloor $, we have,
\[
\lambda_1(g_{T+s} u_x \mathcal{R}^{n+1}) \leq \|g_{T+s} u_x \bv_k\| = q^{-ns} \|g_t u_x \bv_k\| \leq q^{-ns + R_0M_k}.
\]
Using Minkowski's theorem in the function field, we get the relation $ \lambda_2 \geq \lambda_1^{-1} \lambda_{n+1}^{-(n-1)} $. From that, we obtain
\[
\lambda_2(g_{T+s} u_x \mathcal{R}^{n+1}) \geq q^{ns - R_0M_k} q^{-n(n-1)M_k - (n-1)s} \geq q^{s - 2R_0M_k}.
\]
For $ ns > 3R_0M_k $, this implies,
\[
\lambda_2(g_{T+s} u_x \mathcal{R}^{n+1}) \geq q^{-2R_0M_k} > q^{-ns + R_0M_k} \geq \|g_{T+s} u_x \bv_k\|.
\]
Thus, for $ t \in [t_k^-, t_k^x] \cap \mathbb{Z} $, the vector $ \bv_k $ achieves the first minimum of $ g_t u_x \mathcal{R}^{n+1} $.

Additionally, at time $ \lfloor t_k^x \rfloor = T + s_k^x $, we have,
\[
\frac{\lambda_2(g_{\lfloor t_k^x\rfloor} u_x \mathcal{R}^{n+1})}{\lambda_1(g_{\lfloor t_k^x\rfloor} u_x \mathcal{R}^{n+1})} \geq q^{(n+1)s_k^x - 3R_0M_k} > q^{(n+1)(R_1 + R_2)M_k}.
\]
Therefore, $ \bv_k $ continues to achieve $ \lambda_1(g_t u_x \mathcal{R}^{n+1}) $ in the interval $ [t_k^x, t_k^x + (R_1 + R_2)M_k] \cap \mathbb{Z} $, which contains the interval $ [t_k^x, t_k^+] \cap \mathbb{Z}$.

Indeed, for $ s \in [0, (R_1 + R_2)M_k] \cap \mathbb{Z} $, we have,
\begin{align*}
\lambda_2(g_{\lceil t_k^x \rceil + s} u_x \mathcal{R}^{n+1}) \geq &q^{-ns} \lambda_2(g_{\lceil t_k^x \rceil} u_x \mathcal{R}^{n+1})&\\
&\geq q^{-ns} q^{(n+1)(R_1 + R_2)M_k} \lambda_1(g_{\lceil t_k^x \rceil} u_x \mathcal{R}^{n+1})\\
&\geq q^{s} \|g_{\lceil t_k^x \rceil} u_x \bv_k\| \geq \|g_{\lceil t_k^x \rceil + s} u_x \bv_k\|.
\end{align*}

Now, we show that $ \text{A}_{l_k^+} $ holds. We need to verify that for $ t \in [t_k^+, l_k^+M] \cap \mathbb{Z} $, $ c_x(t) > M_{k+1} $. For $ t \in [t_k^+, l_k^+M] \cap \mathbb{Z} $, we compute,
\[
c_x(t) \geq c_x(\lceil t_k^x\rceil) + (t - \lceil t_k^x\rceil) \geq r_\psi(t_k) - 5R_0M_k + (t - t_k)+C_1,
\]
where $C_1$ is some constant independent of $t$. Hence, it simplifies to,
\[
\geq M_{k+1} - R_1M_k - 5R_0M_k + R_2M_k+C_1 > M_{k+1} ,
\]
since we choose $ R_2 > R_1 + 6R_0+C_1 $.

Thus, the condition is satisfied, and we conclude that $ \text{A}_{l_k^+} $ holds.
\end{proof}
\end{description}

\subsection{Branching and Hausdorff dimension}
We now proceed with determining the Hausdorff dimension of $K_{\infty}$.  
To achieve this, it is essential to establish a lower bound on the branching behavior of the Cantor set at each level $l$ within the interval $[t_{k-1}^+, t_k^-] \cap \mathbb{Z}$. This will follow from the Lemma \ref{Large Branching} below, as outlined in \cite{bandi2023hausdorff}. The proof relies on a modified form of the Simplex Lemma, tailored to the ultrametric framework, which was initially introduced in the work of Davenport and Schmidt \cite[page 57]{Sch}.

\begin{lemma}\label{Large Branching}
There exists a constant $ R_3 $ dependent on $ n $ such that for sufficiently large $ k $, if $ l_{k-1}^+ < l \leq l_k^- $ and $ C $ is an element of $ K_{l-1} $, then
\[
\card K_l(C) \geq N^n - R_3 N^{n - \frac{1}{n+1}}.
\]
\end{lemma}

\begin{proof}
Let $ t = \lfloor (l-1)M \rfloor $, where $ M $ is a constant. Given the assumptions $ l_{k-1}^+ < l \leq l_k^- $ and that $ C \subseteq K_{l-1} $, it follows that for all $ x \in C $,
\[
\lambda_1(g_t u_x \mathcal{R}^{n+1}) > q^{-M_k + M}.
\]
Consider a point $ x_0 \in C $ and define the set
\[
S_C \df \left\{ \mathbf{v} \in \mathcal{R}^{n+1} : \| g_t u_{x_0} \mathbf{v} \| < q^{-M_k + (2n-1)M} \right\}.
\]
We now claim that there exists a hyperplane $ H_C \subseteq \K^{n+1} $ such that $ S_C \subseteq H_C $. If no such hyperplane existed, we would encounter linearly independent vectors $ \mathbf{v}_1, \dots, \mathbf{v}_{n+1} \in \mathcal{R}^{n+1} $ satisfying
\[
\| g_t u_{x_0} \mathbf{v}_i \| < q^{-M_k + (2n-1)M} \quad \text{for} \quad 1 \leq i \leq n+1.
\]
However, the lattice $ g_t u_{x_0} \mathcal{R}^{n+1} $ has covolume one, and it is impossible for it to contain $ n+1 $ linearly independent vectors with norms smaller than 1. This contradiction establishes the existence of such a hyperplane $ H_C $, provided $ M_k > (2n-1)M $. Let $e_1=(1,0,\cdots,0).$

Define the set
\[
A_C = \left\{ x \in C : d([e_1], g_t u_x H_C) \leq q^{-M} \right\}.
\]
For any point $ x \in C \setminus A_C $ and any non-zero vector $ \mathbf{v} \in \mathcal{R}^{n+1} $, if $ \mathbf{v} \in S_C $, then $ \mathbf{v} \in H_C $. Consequently, for every nonzero vector $
w \in g_t u_x H_C,
$
the first coordinate satisfies
\[
|(w)_1| \geq q^{-M} \|w\|.
\]

Applying this to
$
w = g_t u_x v,
$
we obtain
\[
|(g_t u_x v)_1| \geq q^{-M} \|g_t u_x v\|.
\] Since the first coordinate is multiplied by $q^M$ under the action of $g_M$, we obtain,
\[
\| g_{t+M} u_x \mathbf{v} \| \geq q^M q^{-M} \| g_t u_x \mathbf{v} \| \geq q^{-M_k + M}.
\]
If $ \mathbf{v} \notin S_C $, we can bound
\[
\| g_{t+M} u_x \mathbf{v} \| \geq q^{-nM} \| g_t u_x \mathbf{v} \| \geq q^{-2nM} \| g_t u_{x_0} \mathbf{v} \| \geq q^{-M_k + M}.
\]
Therefore, for all $ x \in C \setminus A_C $, we have
\[
c_x([lM]) \geq -M_k + M.
\]

 We finally show that the set $ A_C $ is contained within a neighborhood of size $ N^{-l} q^{-M} $ around a hyperplane in $ \K^n $. For any $ x \in C $, define $ y = q^{(n+1)lM} (x - x_0) $ in $ \mathcal{O}^n $, so that
\[
g_t u_x = u_y g_t u_{x_0}.
\]
Consider a linear functional $\phi_{t,C}$ of norm one that is zero on the hyperplane $H_{t,C} = g_t u_{x_0} H_C$. Then, we have
\[
d([e_1], g_t u_x H_C) = d([e_1], u_y H_{t,C}) \asymp d([u_{-y} e_1], H_{t,C}) \asymp \phi_{t,C}(u_{-y} e_1).
\]
Thus, if $ x \in A_C $, we obtain $ \phi_{t,C}(u_{-y} e_1) \lesssim q^{-M} $, implying that $ y $ lies within a neighborhood of size $ O(q^{-M}) $ of the affine hyperplane in $ \K^n $ defined by
\[
\phi_{t,C}(e_1 - y_1 e_2 - \dots - y_n e_{n+1}) = 0.
\]
This means that $ x $ lies in a neighborhood of size $ O(N^{-l} q^{-M}) $ around an affine hyperplane in $ \K^n $. The number of subcubes of $ C $ that intersect this neighborhood is bounded by
\[
\lesssim_n N^n q^{-M},
\]
and thus we obtain
\[
\card K_l(C) \geq N^n (1 - O_n(q^{-M})) = N^n - O_n(N^{n - \frac{1}{n+1}}).
\]
\end{proof}
Now, we proceed to the proof of Theorem \ref{main1}.
\begin{proof}[Proof of Theorem \ref{main1}]
The proof closely follows the approach employed in \cite{bandi2023hausdorff}; however, we include it here for the sake of completeness. To obtain a lower bound for the Hausdorff dimension of $K_{\infty}$, we apply the Mass Distribution Principle. This argument follows identically from the work of \cite{bandi2023hausdorff}, but we include the proof here for the sake of completeness.

\color{black}
For $l\geq 1$, define
\[
b_l = 
\left\{
\begin{array}{ll}
\lfloor N^n(1-R_3q^{-\frac{M}{n}})\rfloor & \mbox{if}\ l_{k-1}^+ < l \leq l_k^-\\
1 & \mbox{if}\ l_k^- < l \leq l_k^+.
\end{array}
\right.
\]
We also replace $K_{\infty}$ by a Cantor subset $F_{\infty}$ that is more regular in nature. Removing some cubes in $K_l$ at each step, we obtain the subset $F_\infty\subset K_\infty$ given as
\[
F_\infty = \bigcap_{l\geq 1}F_l,
\]
where each cube $C$ in $F_{l-1}$ contains exactly $b_l$ subcubes in $F_l$. As stated in Proposition 1.7 of \cite{Falconer}, there is a probability measure $\mu$ that is supported on $F_\infty$, and for each cube $C$ at level $l$ (with $C \subset F_l$), we have the relationship
\[
\mu(C) = \frac{1}{b_1 b_2 \dots b_l}.
\]
We propose that if $\alpha$ is smaller than the limit inferior
\[
\liminf_{l \to \infty} \frac{\log(b_1 b_2 \dots b_l)}{l \log N},
\]
then there exists a constant $C = C_{n, N, \alpha}$ such that for every $x \in \K$ and for all radii $r > 0$,
\[
\mu(B(x,r)) \leq C r^\alpha.
\]
To justify this, choose $l$ such that $N^{-l} < r \leq N^{-l+1}$. The ball $B(x, r)$ can intersect no more than $(3N)^n$ cubes from $F_l$. Hence, we have
\[
\mu(B(x, r)) \leq (3N)^n \cdot (b_1 b_2 \dots b_l)^{-1} \leq (3N)^n N^{-l \alpha} \leq (3N)^n r^\alpha
\]
for sufficiently large $l$, or equivalently, sufficiently small $r$, depending on $n$, $N$, and $\alpha$.

Applying the Mass Distribution Principle, this leads to the conclusion that
\[
\dimh F_\infty \geq \alpha.
\]

Next, we have the following identity,
\[
\liminf_{l \to \infty} \frac{\log(b_1 b_2 \dots b_l)}{l \log N} = \lim_{k \to \infty} \frac{\log(b_1 b_2 \dots b_{l_k^+})}{l_k^+ \log N} = \lim_{k \to \infty} \frac{\log(b_1 b_2 \dots b_{l_k^-})}{l_k^+ \log N}.
\]
If the sequences $(l_k^\pm)_{k \geq 1}$ satisfy $l_{k-1}^+ = o(l_k^-)$, which can always be arranged by choosing a sufficiently rapidly increasing sequence $(t_k)$, we obtain a lower bound for the limit,
\[
\lim_{k \to \infty} \frac{l_k^- \log\left\lfloor N^n - R_3 N^{n - \frac{1}{n+1}} \right\rfloor}{l_k^+ \log N}
= \frac{\log\left\lfloor N^n - R_3 N^{n - \frac{1}{n+1}} \right\rfloor}{\log N} \lim_{k \to \infty} \frac{l_k^-}{l_k}
= \frac{\log\left\lfloor N^n - R_3 N^{n - \frac{1}{n+1}} \right\rfloor}{\log N} \frac{n+1}{n \lambda_\psi}.
\]

\details{Recall that $t_k^- = t_k + r_\psi(t_k)$, where $r_\psi(t_k) \sim -\gamma_\psi t_k = -\frac{n \lambda_\psi - n - 1}{n \lambda_\psi} t_k$, so asymptotically $t_k^- \sim \frac{(n+1)t_k}{n \lambda_\psi}$.}

This gives the result
\[
\dimh \Exact(\psi) \geq \dimh F_\infty \geq \frac{\log\left\lfloor N^n - R_3 N^{n - \frac{1}{n+1}} \right\rfloor}{\log N} \frac{n+1}{n \lambda_\psi}.
\]
As $N$ (or equivalently $M$) tends to infinity, we obtain the desired lower bound on Hausdorff dimension,
\[
\dimh \Exact(\psi) \geq \frac{n+1}{\lambda_\psi}.
\]
\end{proof}
Finally, we proceed to the proof of Theorem \ref{existance}.
\begin{proof}[Proof of Theorem \ref{existance}]
When $\lambda_{\psi}<\infty$, Theorem \ref{main1} proves that \[\dimh(\Exact(\psi))\geq \frac{n+1}{\lambda_{\psi}}.\]
And this implies that $\card(\Exact(\psi))$ is uncountable.
\end{proof}

\subsection*{Acknowledgements} The author sincerely thanks Anish Ghsoh and Ravi Raghunathan for their  encouragement and helpful discussions on this project. She also thanks Prasuna Bandi for several helpful discussions on her paper. Finally, she expresses her gratitude to Yann Bugeaud for showing interest in this project and for his comments on the earlier draft of this article. The author was supported by the Prime Minister’s Research Fellowship, ID: 1302639, during the course of this work.
\bibliographystyle{alpha}
\bibliography{mybib}

\end{document}